\documentclass[11pt, a4paper]{article} 
\usepackage{hyperref}
\usepackage[margin=1.0in]{geometry} 

\usepackage{lmodern}

\usepackage[utf8]{inputenc}
\usepackage[english]{babel}

\usepackage{amssymb,amsfonts}

\usepackage{amsmath, amsthm, color, graphicx} 

\definecolor{dullmagenta}{rgb}{0.4,0,0.4}   
\definecolor{darkblue}{rgb}{0,0,0.4}
\definecolor{darkgreen}{rgb}{0,0.6,0}
\definecolor{darkred}{rgb}{0.6,0,0}

\usepackage{mathtools}

\newtheorem {theorem}{Theorem}
\newtheorem {lemma}[theorem]{Lemma}
\newtheorem {proposition}[theorem]{Proposition}
\newtheorem {corollary}[theorem]{Corollary}

\theoremstyle{definition}
\newtheorem {defn}[theorem]{Definition}

\newtheorem {remark}[theorem]{Remark}

\newcommand{\ri}{\right}

\newcommand {\bR}{{\mathbb R}}
\newcommand {\bN}{{\mathbb N}}

\newcommand {\bC}{{\mathbb C}}

\newcommand{\cF}{{\cal F}} %
\newcommand{\bem}{\l(\! \begin{array}}
\newcommand{\eem}{\end{array}\!\ri)}
\newcommand{\bsm}{\left(\begin{smallmatrix}} 
\newcommand{\esm}{\end{smallmatrix}\right)}  

\newcommand{\NN}{\nonumber}

\DeclareMathOperator{\arcsinh}{arcsinh}

\DeclareMathOperator{\spec}{Spec}

\newcommand{\tcR}[1]{{#1}}

\newcommand{\wU}{U}
\newcommand{\wM}{M}

\DeclareMathOperator{\sgn}{sign}
\DeclareMathOperator{\Dom}{Dom}

\renewcommand{\Re}{\operatorname{Re}}
\renewcommand{\Im}{\operatorname{Im}}

\newcommand{\placeholder}[2][]{%
  \setlength{\fboxsep}{-\fboxrule}
  \framebox{\rule{0pt}{100pt}\rule{150pt}{0pt}}
}

\newcommand{\ir}{\mathrm{i}}
\newcommand{\dr}{\mathrm{d}}
\newcommand{\er}{\mathrm{e}}
\usepackage{enumerate}
\renewcommand{\phi}{\varphi}

\title{Accumulation of complex eigenvalues of an indefinite Sturm--Liouville operator with a shifted Coulomb potential}

\author{
Michael Levitin\thanks{{\bf ML: }Department of Mathematics and Statistics, University of Reading, Whiteknights, PO Box 220, Reading RG6 6AX,
UK; m.levitin@reading.ac.uk; \url{http://www.personal.reading.ac.uk/\~ny901965/}}
\and 
Marcello Seri\thanks{{\bf MS: }Department of Mathematics, University College London, Gower Street, London WC1E 6BT, UK; m.seri@ucl.ac.uk; \url{http://academic.mseri.me/}}
}
\date{16 September 2015\\typeset \today}
\renewcommand\footnotemark{}

\begin{document}

\maketitle
\begin{abstract}
For a particular family of long-range potentials $V$, we prove that the eigenvalues of the 
indefinite Sturm--Liouville operator $A = \sgn(x)(-\Delta + V(x))$ accumulate to
zero asymptotically along specific curves in the complex plane. 
Additionally, we relate the asymptotics of complex eigenvalues to the two-term asymptotics
of the eigenvalues of associated self-adjoint operators.
\end{abstract}


\section{Introduction}

Given a real-valued potential $V$ such that
\begin{equation}\label{def:Vclass}
  V \in L^{\infty}(\bR),\quad
    \lim_{x\to \pm\infty} V(x) = 0,\quad
    \limsup_{x\to \pm\infty} x^2 V(x) < -\frac14,
\end{equation}
consider a one-dimensional Schr\"{o}dinger operator in $L^2(\bR)$
\begin{equation}\label{def:T}\begin{split}
  T &:= T_V := -\frac{\dr^2}{\dr x^2} + V(x),  \\
  \Dom(T) &:= \left\{
    f\in L^2(\bR) \mid f,f' \in AC(\bR),\; T f\in L^2(\bR)
  \right\}.
\end{split}\end{equation}
It is well known that in this case the spectrum 
$\spec(T)$ is bounded from below, the essential spectrum $\spec_{\mathrm ess}(T)=[0,\infty)$, and the negative spectrum $\spec(T)\cap(-\infty,0)$ 
consists of eigenvalues accumulating to zero from below.

Let $J := \sgn(x)$ be the multiplication operator by $\pm1$ on $\bR_\pm$.
In what follows we consider the point spectrum of the operator
\begin{equation}\label{def:op}
  A := A_V := JT_V, \quad \Dom(A) = \Dom(T).
\end{equation}
This operator is not self-adjoint (and not even symmetric) on $L^2(\bR)$, and its spectrum need not therefore be real.  However, as $J^*=J^{-1}=J$,  $A$ can be treated as 
a self-adjoint operator in the Krein space 
$(L^2(\bR), [\cdot, \cdot])$ with indefinite inner product 
\[
[f,g] := \langle Jf, g\rangle_{L^(\bR)}=\int_\bR f(x) \overline{g(x)} \sgn(x) \dr x
\]
or equivalently as a $J$-self-adjoint operator \cite{ai}.
Operators of type \eqref{def:op} have been studied both 
in the framework of operator pencils, cf. \cite{davieslevitin, pencils},
and of indefinite Sturm--Liouville problems 
\cite{bkt08, behrndttrunk, karabashtrunk, kreinsp}.

In both settings the literature is extensive, starting mostly with Soviet contributions in the 1960s,  
including those by Krein, Langer,
Gohberg, Pontryagin and Shkalikov. We refer to \cite{pencils, kreinsp} 
for reviews and bibliographies.
In particular due to its many applications, for example in control theory, 
mathematical physics and mechanics,  the field is still very active,
with recent works on the theoretical,as well as numerical, aspects, 
(see e.g. \cite{davieslevitin,elp,htvd,v14} and references therein).

\tcR{In the special case of indefinite Sturm--Liouville operators, it is well known that
for positive potentials, $V\geq 0$, the spectrum of $A_V$ is real and the operator $A_V$ is similar to a self-adjoint operator \cite{cula,cuna,py,ko}. 
At a very basic level, this can be seen from the following abstract construction: if $R$ and $S$ are self-adjoint operators with $R>0$, then, under mild restrictions,
the spectrum of $R^{-1}S$ is the same as the spectrum of the self-adjoint operator $R^{-1/2}SR^{-1/2}$, and is therefore real.}

\tcR{The case $V\in L^1(\bR, (1+|x|)\dr x)$ has been considered in \cite{kakoma}, where it is shown that $A_V$ is  self-adjoint  iff $\spec(A_V) = \bR$.
Finally, for (quasi-)periodic finite zone potentials,  \cite{kama} explores some conditions under which $A_V$ is similar to a self-adjoint or a normal operator. For a review of indefinite weighted Sturm--Liouville  problems on a finite interval, see \cite{fl}.}

\tcR{Let us return to our original problem \eqref{def:op}. Recently, there was a rapid growth of interest in the case of non-positive potentials, especially by Behrndt, Trunk, and collaborators \cite{b07,bkt08,behrndttrunk,bpt}, clarifying the structure of their spectra and other properties as well as stating new conjectures on rather unusual spectral behaviour \cite{b13}.}

The following known results are a particular case of 
\cite[Theorem 1 and Theorem 2]{bkt08} and \cite[Theorem 4.2]{bpt}. 
\begin{proposition}\label{prop:one} For the operator $A$ in \eqref{def:op},
\begin{enumerate}[{\bf (a)}]
  \item $\spec(A)$ is symmetric with respect to $\bR$.
  \item $\spec_{\mathrm ess}(A) = \bR$.
  \item $\spec(A)\setminus{\bR}$ consists of eigenvalues of finite multiplicity.
  \item No point of $\overline{\bR}\setminus\{0\}$ is an accumulation point of 
    non-real eigenvalues of $A$.
  \item At least one of the following statements is true:
  \begin{enumerate}[{\bf (i)}]
    \item The non-real eigenvalues of $A$ accumulate only to $0$;
    \item There exist embedded eigenvalues of $A$ in $\bR_+$ that accumulate to $0$;
    \item There exist embedded eigenvalues of $A$ in $\bR_-$ that accumulate to $0$;
    \item The growth of $\lambda \mapsto (A - \mu)^{-1}$ near zero is not of 
      finite order.
  \end{enumerate} 
  \item If additionally $V$ is even, $V(x)=V(-x)$, then $\spec(A)$ is also symmetric with respect to $\ir\bR$.
  \item The non-real spectrum of $A$ is contained in the strip $|\Im\mu|<2\|V\|_\infty$.
\end{enumerate}
\end{proposition}

Despite the amount of information on the structure of the spectrum of $A$, known proofs of Proposition \ref{prop:one} are not constructive and, in fact, we do not even know \emph{a priori} which of the four statements (e)(i)--(iv) are true for a particular given potential.  
Some numerical experiments, cf. \cite{bkt08}, have recently led to conjecture that 
statement (e)(i) in Proposition 1 may hold for many potentials satisfying \eqref{def:Vclass}, see Figure \ref{fig:example1}. 

\begin{figure}[htb!]
\centerline{\includegraphics[width=0.8\textwidth]{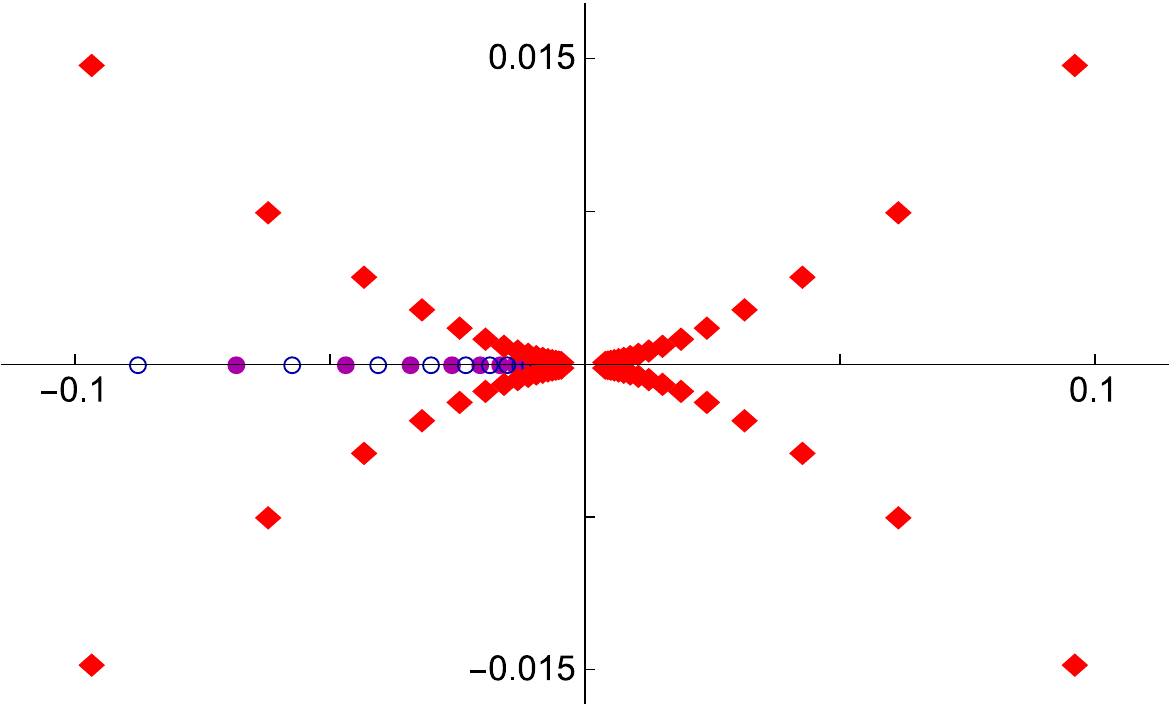}}
\caption{A numerical example showing accumulation to $0$ of complex eigenvalues (red diamonds)
of the operator $A_\gamma$, $\gamma=2.5$. 
The magenta and  white circles on the negative real axis are the eigenvalues 
of $T_\gamma$ corresponding to the eigenfunctions which are even or odd with respect to zero, cf.~\cite{bkt08}.\label{fig:example1}}
\end{figure}

In this paper we prove that for a particular family of potentials 
\begin{equation}\label{def:V}
V(x) = V_\gamma(x) = - \frac \gamma{1+|x|}, \quad \gamma>0,
\end{equation}
Proposition 1(e)(i) holds. Moreover we also prove (Theorem \ref{thm:1}) that the complex eigenvalues of 
\[
T=T_\gamma := T_{V_\gamma}
\]
accumulate to zero asymptotically along specific curves in the complex plane, and that the explicit asymptotics of 
complex eigenvalues of $T_\gamma$ can be obtained from the asymptotics of eigenvalues of the self-adjoint operator
\begin{equation}\label{def:Agamma}
A=A_\gamma := A_{V_\gamma}
\end{equation}
(or, more precisely, from the eigenvalues of its restriction on either even or odd (with respect to zero) subspace). We also 
extend these results to the more general non-symmetric potentials 
\begin{equation}\label{def:Vnonsym}
\quad V_{\gamma_-,\gamma_+}(x) = \begin{cases}
\displaystyle-\frac{\gamma_+}{1+|x|} & \mbox{if } x > 0\\
\displaystyle-\frac{\gamma_-}{1+|x|} & \mbox{if } x < 0
\end{cases}, \quad  \gamma_+, \gamma_- \in \bR_+.
\end{equation}

The rest of the paper is organised as follows. Section 2 contains the statements of our main results. The proofs, as well as some numerical examples,  are in Sections 3--5; they are based on the explicit expressions for Jost solutions of the differential equation 
\[
-\frac{\dr^2}{\dr x^2}g(x)-\frac{\gamma}{1+x}g(x)=\mu g(x)
\]
on $\bR_+$ and a rather delicate asymptotic analysis involving Kummer functions. A brief exposition of some auxiliary results, mainly due to Temme \cite{temme}, which we use in our proofs, is given in the Appendix.

\section{Sharp asymptotics of the eigenvalues of the self-ajdiont operator}
Let $T^{D}_\gamma$ and $T^{N}_\gamma$ denote the restrictions of the operator 
$T_{\gamma}$ to $\bR_+$ with Dirichlet 
and Neumann boundary condition at zero, resp. By the spectral theorem, for symmetric potentials $V_\gamma(x)$
\[
\spec(T_\gamma)=\spec(T^D_\gamma)\cup\spec(T^N_\gamma)
\]
with account of multiplicities.
Let $-\lambda^\#_n(\gamma)$  denote the eigenvalues 
of $T^\#_\gamma$, $\#={}D\text{ or }N$, ordered increasingly.  
In what follows we often  drop the explicit dependence on $\gamma$. 

It is well-known that $-\lambda_n^\#<0$ and $-\lambda_n^N<-\lambda_n^D<-\lambda_{n+1}^N$ for all $n\in\bN$, and also that $-\lambda_n^\#\to0-$ as $n\to\infty$.

Before stating our main results, we need some additional notation.

\begin{defn}\label{defn:Theta} Let $\mathcal{F}$ denote the class of piecewise smooth functions $F:\bR_+\to\bR$ which have a discrete set of singularities (with no finite accumulation points). At each singularity  both one-sided limits of $F$ are  $\pm\infty$ and differ by sign. Assume for simplicity that $0$ is not a singularity of $F$, and that $F(0)=0$. For $F\in\cF$ we denote by 
$\Theta_F(x)$ the continuous branch of the multi-valued $\operatorname{Arctan}(F(x))$ such that $\Theta_F(0)=0$.
\end{defn}

\begin{remark} Away from the singularities of $F$, the function $\Theta_F(x)$ can be written in terms of the ordinary $\arctan(F(x))$ (which takes the values in $\left[-\frac{\pi}{2}, \frac{\pi}{2}\right]$) and 
 the {\em total signed index} of $F$ on $[0,x]$, which we denote by $Z_F(x)$, and which is defined as the total number of jumps from $+\infty$ to $-\infty$ on $[0,x]$ minus the total number of jumps in the opposite direction:
  \begin{equation}\label{eq:Z}
 Z_F(x):=\left(\sum_{\{\tau\in(0,x]\mid \lim\limits_{t\to \tau-}F(t)=+\infty\}}-\sum_{\{\tau\in(0,x]\mid \lim\limits_{t\to \tau-}F(t)=-\infty\}}\right)\,1.
 \end{equation}
 Then
  \begin{equation}\label{eq:Theta}
 \Theta_F(x)=\arctan(F(x))+\pi Z_F(x).
 \end{equation}
 \end{remark}

 Obviously, $\Theta_{-F}(x) = -\Theta_F(x)$.

Our first result gives sharp two-term asymptotics of eigenvalues  (accumulating to zero) of the self-adjoint operators $T_\gamma^\#$.

\begin{theorem}\label{thm:SAevs}
As $n\to \infty$, 
\begin{align*}
\lambda^D_n(\gamma) &= \frac{\gamma^2}{4 n^2} \left( 1
  - \frac{2}{\pi n} \Theta_{R_1}(\gamma) 
  + O\left(\frac1{n^2}\right) \right), \\
\lambda^N_n(\gamma) &= \frac{\gamma^2}{4 n^2} \left( 1
  - \frac{2}{\pi n}\Theta_{R_0}(\gamma) 
  + O\left(\frac1{n^2}\right) \right),
\end{align*}
where 
\[
R_k(\gamma)=\frac{J_k\left(2\sqrt{\gamma}\right)}{Y_k\left(2\sqrt{\gamma}\right)},
\]
and
$J_k$ and $Y_k$ denote the Bessel functions of the first and second kind, respectively. 
\end{theorem}

\begin{figure}[htb!]
\centerline{
\includegraphics[width=.7\linewidth]{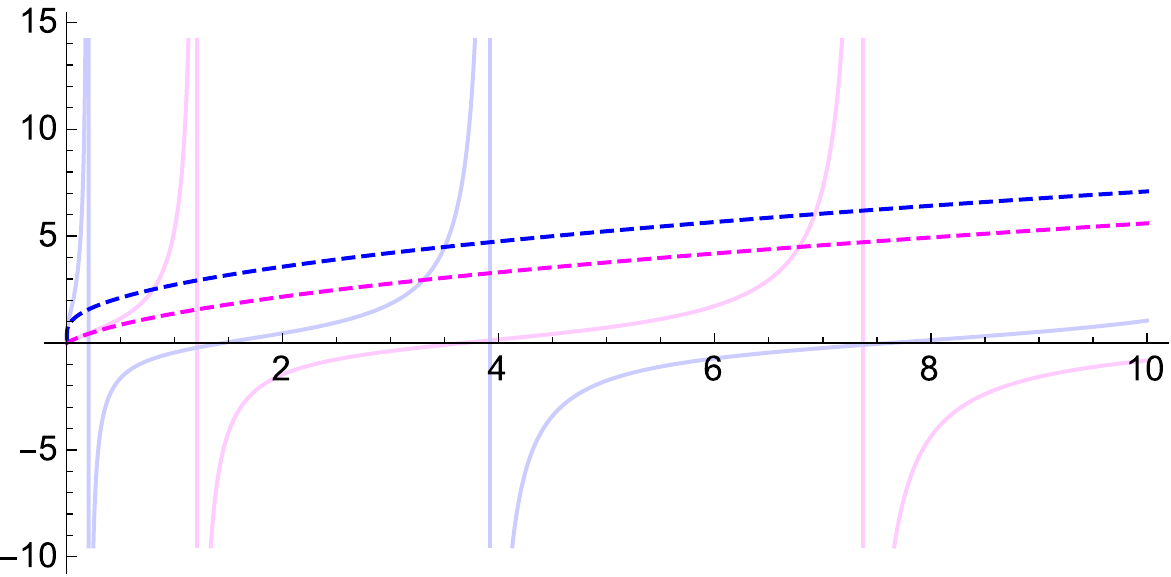}}
\caption{$R_1(\gamma)$ and $R_0(\gamma)$ (resp. magenta and blue line) from Theorem \ref{thm:SAevs} and the corresponding $\Theta_{R_1}(\gamma)$ and $\Theta_{R_0}(\gamma)$ (resp. magenta and blue dashed line).\label{fig:Theta}}
\end{figure}

\begin{figure}[htb!]
\centerline{
\includegraphics[width=.66\linewidth]{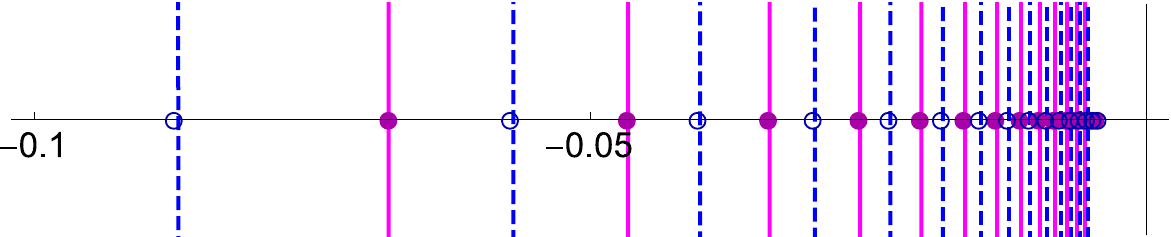}}
\caption{Approximate eigenvalues of $T_\gamma$, $\gamma=2.5$, as described in Theorem \ref{thm:SAevs}. The magenta (resp. white) circles correspond to $-\lambda_n^D$ (resp. $-\lambda_n^N$). 
The solid (resp. dashed) lines are the left-hand sides of \eqref{eq:ev-dirichlet-sa} (resp. \eqref{eq:ev-neumann-sa}) whose roots are the eigenvalues.
Only curves and eigenvalues up to $-0.01$ are displayed.\label{fig:sa}
}
\end{figure}

This immediately implies

\begin{corollary}\label{cor:diff} As $n\to\infty$,
\[
\lambda_n^D-\lambda_n^N=O\left((\lambda_n^\#)^{3/2}\right).
\]
\end{corollary}

Some old papers found both in physical and mathematical literature addressed the problem of approximating the eigenvalues of 
Schr\"odinger operators with shifted Coulomb potentials, see e.g. \cite{geszt,vhareingen} and references
therein. However they were considering somewhat different asymptotic limits, and  to the best of our knowledge the two-terms 
asymptotics of Theorem \ref{thm:SAevs} are new. 

\section{Sharp asymptotics of the eigenvalues of the non-self-ajdiont operator}

Our main result is the following

\begin{theorem}\label{thm:1}
\begin{itemize}
\item[{\bf (i)}] The eigenvalues of $A_\gamma$ lie asymptotically on the curves
\begin{equation}\label{eq:Im_mu}
\left|\Im \mu\right| = \Upsilon(\gamma)\,|\Re \mu|^{3/2} + O\left((\Re \mu)^2\right),\qquad\mu\to 0,
\end{equation}
where
\begin{equation}\label{eq:Ups}
\Upsilon(\gamma)=\frac{1}{\pi\gamma}\log\frac{q^2(\gamma)}{1+q^2(\gamma)}
\end{equation}
and
\begin{equation}\label{eq:q}
q(\gamma):=\pi\sqrt{\gamma}\left( 
    J_0(2\sqrt{\gamma})\, J_1(2\sqrt{\gamma}) + Y_0(2\sqrt{\gamma})\,Y_1(2\sqrt{\gamma})
  \right).
\end{equation}

\item[{\bf (ii)}] More precisely, the eigenvalues  $\{\mu\}_{n\in\bN}$ of $A_\gamma$ in the first quadrant (ordered by decreasing real part) are related to the absolute values  $\lambda_n^\#$ of the eigenvalues of the self-adjoint operators $T_\gamma^\#$, $\#=D$ or $N$, by
\begin{equation}\label{eq:mu_first_quadrant}
\mu_n=\lambda_n^D+\Upsilon^-(\gamma)(\lambda_n^D)^{3/2}+O\left((\lambda_n^D)^{2}\right)=\lambda_n^N+\Upsilon^+(\gamma)(\lambda_n^N)^{3/2}
+O\left((\lambda_n^N)^{2}\right)
\end{equation}
as $n\to\infty$, where
\begin{equation}\label{eq:Ups^mp}
\Upsilon^\mp(\gamma)=\frac{4}{\pi\gamma}\arctan\left(\frac{1}{\ir\mp 2q(\gamma)}\right).
\end{equation}
The expressions for eigenvalues in the other quadrants are obtained by symmetries with respect to 
$\bR$ and $\ir\bR$.
\end{itemize}
\end{theorem}

Before proceeding to the proofs, we want to discuss  the statements of  Theorem \ref{thm:1} in more detail. 

\begin{remark}\label{rem:comments}
\begin{itemize}
\item[(a)] It is immediately seen from \eqref{eq:Ups} and \eqref{eq:Ups^mp} that
\[
\Upsilon(\gamma)=\Im\Upsilon^-(\gamma)=\Im\Upsilon^+(\gamma).
\]
\item[(b)] If we introduce two functions $\tau^\mp:\bR_+\to\bC$ by
\[
\tau^\mp(t)=t+\Upsilon^\mp t^{3/2},
\]
then 
\[
\Im\tau^-(t)=\Im\tau^+(t)=\Upsilon t^{3/2},
\]
which is just another way of writing (a). (We have dropped the dependence on $\gamma$ here for clarity.)
\item[(c)] The statement (ii) of  Theorem \ref{thm:1} contains in fact several results. Taking the imaginary parts of \eqref{eq:mu_first_quadrant} leads immediately to 
\eqref{eq:Im_mu} by Corollary \ref{cor:diff}. The equalities obtained by taking the real parts  \eqref{eq:mu_first_quadrant} are more intricate: they indicate that, up to the terms of order $(\lambda_n^\#)^2$, the values of $\Re(\tau^-(\lambda_n^D))$ and $\Re(\tau^+(\lambda_n^N))$ coincide. 

In other words, we can construct the eigenvalues of the non-self-adjoint operator $A_\gamma$ by perturbing either the anti-symmetric self-adjoint eigenvalues $\lambda_n^D$ or the symmetric self-adjoint eigenvalues $\lambda_n^N$, and the different formulae still lead to the same result, modulo higher-order terms.

\item[(d)] The only previously known bound, see Proposition 1(g), implies only that for our potential $V_\gamma$, $|\Im\mu|<2\gamma$.
\end{itemize}
\end{remark}

The typical eigenvalue behaviour is illustrated in Figure \ref{fig:curves}.

\begin{figure}[h!]
\centerline{
\includegraphics[width=.66\linewidth]{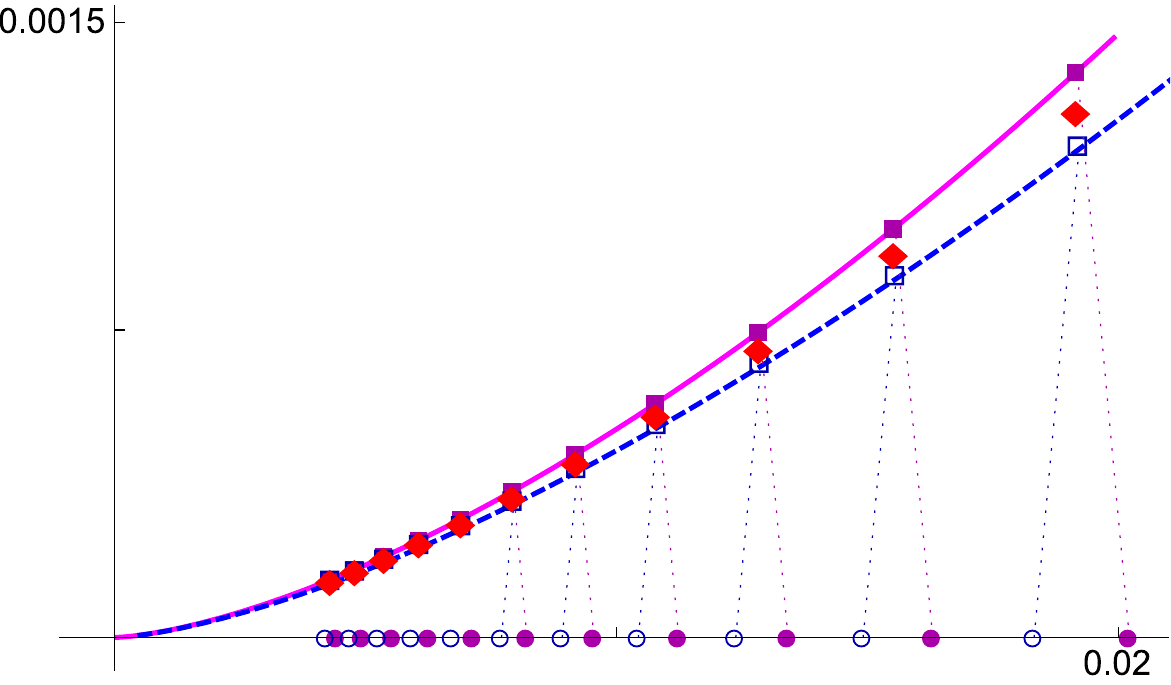}}
\caption{The red diamonds are numerically computed exact eigenvalues of $A_\gamma$, $\gamma=2.5$, lying in the first quadrant. The complex parametric curves 
$\mu=\tau^-(t)$ (the magenta solid line) and  $\mu=\tau^+(t)$ (the blue dashed line) are as in Remark \ref{rem:comments}(b). The approximated complex eigenvalues, 
computed by the first part of formula \eqref{eq:mu_first_quadrant} are shown as magenta squares, and computed by the second part of formula \eqref{eq:mu_first_quadrant} are shown as white squares. The absolute values of the eigenvalues of $T_\gamma$ are marked on the real line in the same way as in Figure \ref{fig:sa}. The dotted arrows are to indicate which real eigenvalue ``produces'' the corresponding complex ones.\label{fig:curves}}
\end{figure}

\section{Explicit form of the Jost solutions}

\subsection{Solutions of the equation on the half-line}

Let $f_\xi$ be a general solution of the differential equation 
\begin{equation}\label{eqn:SAtrafo-y}
  -\frac{\dr^2}{\dr y^2} f(y) 
  - \frac 1{y} f(y) 
= \xi f(y), \qquad y\in\bR_+.
\end{equation}

The change of variables 
\begin{equation}\label{change:y}
  y = \gamma(x+1), 
\end{equation}
relates $f_\xi$ with solutions $g_{\mu,\gamma}$ 
of the differential equation
\begin{equation}\label{eqn:SAorigK}
-\frac{\dr^2}{\dr x^2} g(x) - \frac \gamma{1+x} g(x) = \mu\; g(x), 
  \qquad x\in \bR_+,
\end{equation}
by 
\begin{equation}\label{eqn:gmu-from-flambda}
g_{\mu,\gamma}(x) = f_{\mu/\gamma^2}\left(\gamma(x+1)\right).
\end{equation}

\subsection{Explicit solutions of the differential equation \eqref{eqn:SAorigK}}

Assuming $\xi\in\overline\bC_+$, we will write
\begin{equation}\label{def:ell}
s := -\ir \sqrt{- \xi}
\end{equation}
where we  take the principal branch of $\sqrt{\;\cdot\;}$, i.e. the uniquely determined 
analytic branch that maps $\bR_+$ into itself. 
Obviously $s^2 = \xi$.

With the ansatz
\begin{equation}\label{eq:ansatz-for-tildeg}
f(y) = y \er^{-\ir s y} h(y),
\end{equation}
equation \eqref{eqn:SAtrafo-y} can be reduced to
\begin{equation}\label{eq:SAeqny1-afteransatz}
 - y \frac{\dr^2}{\dr y^2} h(y) 
 -(2-2\ir s y) \frac{\dr}{\dr y} h(y) 
 + (2\ir s -1) \, h(y)
 = 
 0.
\end{equation}
With the additional change of variables
\begin{equation}\label{change:w}
w = 2 \ir s y
\end{equation}
we arrive at a particular case of the Kummer Hypergeometric Equation 
\cite[Chapter 13.1]{nist}
\begin{equation}\label{eqn:kummerhg}
  w \frac{\dr^2}{\dr w^2} \tilde h(w) + (b-w) \frac{\dr}{\dr w} \tilde h(w) - a \tilde h(w) = 0
\end{equation}
with
\[
a = 1 - \frac1{2 \ir  s }, \qquad
b = 2, \qquad
w = 2\ir s y, \qquad
\tilde h(w) = h \left( \frac{w}{2\ir s} \right).
\]
The two linearly independent solutions of \eqref{eqn:kummerhg} are known as Kummer hypergeometric 
functions $M(a,b; w)$ and $U(a,b; w)$ so that 
\[
 \tilde h(w) = C_1 U(a,b; w)+C_2 M(a,b; w),\qquad C_1,C_2=\operatorname{const}.
 \]

Hence the solutions $f_\xi(y)$ of \eqref{eqn:SAtrafo-y} are of the form
\begin{equation}\label{sol:gtilde-ysk}
f_\xi(y) = y \er^{-\ir s y} \left( 
  C_1 \wU\left(1 - \frac1{2 \ir s }, 2;\; 2\ir s  y\right)
  + C_2 \wM\left(1 - \frac1{2\ir s }, 2;\; 2\ir s y\right)
\right).
\end{equation}

Thus, by \eqref{eqn:gmu-from-flambda}, 
the solutions of \eqref{eqn:SAorigK} are of the form
\[
\begin{split}
g_{\mu,\gamma}(x) ={}& \gamma (1+x) \er^{-\sqrt{-\mu}\, (1+x)} \Bigg( 
  C_1 \wU\left(1 - \frac \gamma{2 \sqrt{-\mu}}, 2;\; 2\sqrt{-\mu}\, (1+x)\right)\\
  &+ C_2 \wM\left(1 - \frac \gamma{2 \sqrt{-\mu}}, 2;\; 2\sqrt{-\mu}\, (1+x)\right)
\Bigg).
\end{split}
\]

\subsection{The Jost solutions of \eqref{eqn:SAorigK}}
It is well known, see e.g. \cite[13.7.1 and 13.7.2]{nist}, 
that the first order asymptotic behaviour of
the Kummer Hypergeometric Functions is given, as $|w|\to\infty$, by
\begin{align*}
\wU(a,b;w) &\sim w^{-a}, 
  &\qquad - \frac{3\pi}2 < &\arg(w) < \frac{3\pi}2, \\
\wM(a,b;w) &\sim \frac{\er^w w^{b-a}}{\Gamma(a)} + \frac{\er^{\pi \ir a} w^{-a}}{\Gamma(b-a)}, 
  &\qquad -\frac{\pi}2 \leq &\arg(w) < \frac{3\pi}2, 
  \quad a, b-a\not\in -\bN\cup\{0\},
\end{align*}
where $\Gamma(\cdot)$ stands for the usual Gamma function.

For $\mu\in\bC\setminus\bR_+$, we have $-\sqrt{-\mu} \subset \{z\in\bC \mid \Re z < 0\}$, and therefore 
\[
 \wU\left(1 - \frac \gamma{2 \sqrt{-\mu}}, 2;\; 2 y \sqrt{-\mu}\right) 
   \sim \left(2 y \sqrt{-\mu}\right)^{\frac \gamma{2 \sqrt{-\mu}} - 1},
\]
and
\[
 \wM\left(1 - \frac \gamma{2 \sqrt{-\mu}}, 2;\; 2 y \sqrt{-\mu}\right)
   \sim \frac{
       \left(2 y \sqrt{-\mu}\right)^{\frac \gamma{2 \sqrt{-\mu}} + 1}
     }{\Gamma\left(1 - \frac \gamma{2 \sqrt{-\mu}}\right)}
     \er^{2 y \sqrt{-\mu}}
\]
as $y \to \infty$. 

This in turn implies that the $f_\xi$ and $g_\mu$ defined above are $L^2(\bR_+)$ 
if and only if $C_2 = 0$.  For convenience, we choose further on the normalisation $C_1 = 1$.

The $L^2(\bR_+)$ solutions of  \eqref{eqn:SAorigK} are called the \emph{Jost solutions}. We denote them by
\begin{equation}\label{def:jost}
\varphi_\mu(\gamma,x) := \left.g_{\mu,\gamma}(x)\right|_{C_1=1, C_2=0}= \gamma (1+x) \er^{-\sqrt{-\mu}\, (1+x)}\; 
  \wU\left(1 - \frac \gamma{2 \sqrt{-\mu}}, 2;\; 
    2 (1+x) \sqrt{-\mu}\right).
\end{equation}

\section{Proof of Theorem \ref{thm:SAevs}}
\subsection{The characteristic equations for the self-adjoint problem}

It is well known that $T$ has a negative discrete spectrum accumulating to zero.
We need a more detailed knowledge of the asymptotics of its eigenvalues
and of the corresponding eigenfunctions.

It follows from the arguments of the previous sections that up to a scaling constant the eigenfunctions of the self-adjoint problem \eqref{eqn:SAorigK} 
are given, on $\bR_+$, by
\[
\psi_\lambda(\gamma,x) = \gamma (1+x) \er^{-\sqrt{\lambda}\, (1+x)}\; 
  \wU\left(1 - \frac \gamma{2 \sqrt{\lambda}}, 2;\; 
    2 (1+x) \sqrt{\lambda}\right).
\]

The eigenvalues $-\lambda_n^D$ of the self-adjoint operator $T_\gamma^D$ with Dirichlet boundary conditions at zero
are thus given by the solutions of $\psi_\lambda(\gamma,0)=0$, i.e.
\begin{equation}\label{eq:ev-dirichlet-sa}
 \frac{\gamma\, \er^{-\sqrt{\lambda}} }{2\sqrt{\lambda}} 
   \wU\left(-\frac{\gamma}{2 \sqrt{\lambda}}, 0, 2 \sqrt{\lambda}\right) = 0
\end{equation}

The eigenvalues $-\lambda_n^N$ of the self-adjoint operator $T_\gamma^N$ with Neumann boundary conditions at zero
are 
given by the solutions of $\frac{d}{dx}\psi_\lambda(\gamma,x)\Big|_{x=0} = 0$, i.e.
\begin{equation}\label{eq:ev-neumann-sa}
 \frac{\gamma\, \er^{-\sqrt{\lambda}} }{\lambda^{2}} 
 \left(
   (\gamma - 2 \sqrt{\lambda}) 
   \wU\left(-\frac{\gamma}{2 \sqrt{\lambda}}, -1, 2 \sqrt{\lambda}\right)
   + 2 \sqrt{\lambda}  (\sqrt{\lambda} - 1)
   \wU\left(-\frac{\gamma}{2 \sqrt{\lambda}}, 0, 2 \sqrt{\lambda}\right)
\right) = 0.
\end{equation}

The  solutions of transcendental equations \eqref{eq:ev-dirichlet-sa} and \eqref{eq:ev-neumann-sa} can be computed numerically, although it is a non-trivial task as the left-hand sides of these equations oscillate wildly for small $\lambda$ (cf. Figure \ref{fig:Uosc}). 
We instead use asymptotic techniques to  approximate the Kummer functions as $\lambda\to 0$ and to control their oscillations. A quick took at \eqref{eq:ev-dirichlet-sa} and \eqref{eq:ev-neumann-sa}  shows that we require asymptotic formulas, as $\lambda\to0+$, for
\begin{equation}\label{eq:UtoApproxSA}
  \wU\left(-\frac{\gamma}{2 \sqrt{\lambda }},c;\, 2 \sqrt{\lambda }\right),
  \quad c\in\{0,-1\},
\end{equation}
Unfortunately, it is a difficult task --- the corresponding formulas, are not, in fact, in the standard references. We rely, instead, on the results from the forthcoming book \cite{temme} which we  summarise and adapt in the Appendix. 

\subsection{Asymptotic solutions of a transcendental equation} A crucial element of our analysis is the investigation of the large $\kappa$-roots of the equation 
\begin{equation}\label{eq:tanG}
\tan(\gamma\kappa) =G(\kappa,\gamma)
\end{equation}
where $\gamma$ is treated as a parameter, and where $G$ depends analytically on $\kappa$ in the vicinity of $\kappa=\infty$ and, to leading order,  is of class $\cF$ as a function of $\gamma$.  The required results are summarised in the following

\begin{lemma}\label{lemma:exp}
	Let $G(\kappa,\gamma)$ be an analytic function of $\kappa$ around  $\kappa=\infty$ such that 
	\[
	G(\kappa,\gamma) = G_0(\gamma)\left(1 + O(\kappa^{-1})\right),
	\qquad\text{as}\qquad
	\kappa\to \infty,
	\] 
	$G_0\in\cF$, and the $O$ terms are regular in $\gamma$. 
	Then the solutions $\kappa_n(\gamma)$, ordered increasingly, of the equation \eqref{eq:tanG},
	are given, as $n\to\infty$,  by
\begin{equation}\label{eq:kappagamma}
	\kappa_n(\gamma) = \frac{\pi n}{\gamma} + \frac{1}{\gamma} \Theta_{G_0}(\gamma) + O(n^{-1}).
\end{equation}
\end{lemma}

The proof of Lemma  \ref{lemma:exp} is in fact immediate as soon as we recall Definition \ref{defn:Theta} of $\Theta$ and the fact that $\tan$ is $\pi$-periodic.

Considering additional terms in the expansion of $G$ one can get additional terms in the expansion of $\kappa_n$. This is in fact what we do in more detail in Section \ref{sec:det}. 

\subsection{Approximation of Dirichlet eigenvalues}

We can use the asymptotic approximation obtained in \eqref{eq:approxUexp} to reduce \eqref{eq:ev-dirichlet-sa} to the simpler form
\begin{equation}
\label{eq:ev-d-sa-appr-0}
  \cos\left(\frac{\gamma\pi}{2\sqrt\lambda}\right) \left(J_1\left(2\sqrt{\gamma}\right) + O(\lambda)\right)
+ \sin\left(\frac{\gamma\pi}{2\sqrt\lambda}\right) \left(Y_1\left(2\sqrt{\gamma}\right) + O(\lambda)\right)
= 0.
\end{equation}
This in turn can be rewritten as
\begin{equation}\label{eq:ev-d-sa-appr}
\tan\left(\frac{\gamma\pi}{2\sqrt\lambda}\right)
= 
- \frac{J_1\left(2\sqrt{\gamma}\right)}{Y_1\left(2\sqrt{\gamma}\right)} +O(\lambda).
\end{equation}
Applying Lemma \ref{lemma:exp} with
\[
\kappa=\frac{1}{2\sqrt\lambda},
\qquad G_0(\gamma)=- \frac{J_1\left(2\sqrt{\gamma}\right)}{Y_1\left(2\sqrt{\gamma}\right)} 
  = - R_1(\gamma),
\]
we obtain, after a minor effort,
\begin{align}
\lambda_n &= \frac{\gamma^2 \pi^2}4
  \left( n\pi - \Theta_{G_0}(\gamma) \right)^{-2} + O(n^{-4}) \NN \\
&= \frac{\gamma^2}{4 n^2} \left( 1 + \frac2{\pi n} \Theta_{G_0}(\gamma) + O(n^{-2}) \right) \\
&= \frac{\gamma^2}{4 n^2} \left( 1 - \frac2{\pi n} \Theta_{R_1}(\gamma) + O(n^{-2}) \right)
\label{eq:ln-d-appr}
\end{align}
as $n \to +\infty$, thus proving the first part of Theorem \ref{thm:SAevs}.

\subsection{Approximation of Neumann eigenvalues}

The analysis for Neumann eigenvalues is  slightly more complicated. Again we can use \eqref{eq:approxUexp} to reduce \eqref{eq:ev-neumann-sa} to
\begin{equation}
\label{eq:ev-n-as-approx}
\tan\left(\frac{\gamma\pi}{2\sqrt\lambda}\right)
= -\frac{
  P(\gamma,\lambda)
}{
  Q(\gamma,\lambda)
}
\end{equation}
where
\begin{align*}
P(\gamma,\lambda) :={}& \sqrt{\gamma} (5 \sqrt{\lambda} - 8)(\sqrt{\lambda}+1) J_1\left(2\sqrt{\gamma}\right)
  + (11\sqrt{\lambda} - 8\gamma)(2\sqrt{\lambda}-\gamma) J_2\left(2\sqrt{\gamma}\right) \\
  &-8\sqrt{\gamma}\sqrt{\lambda}(2\sqrt{\lambda}-\gamma) J_3\left(2\sqrt{\gamma}\right)+O(\lambda^{3/2}), \\
Q(\gamma,\lambda) :={}& \sqrt{\gamma} (5 \sqrt{\lambda} - 8)(\sqrt{\lambda}+1) Y_1\left(2\sqrt{\gamma}\right)
  + (11\sqrt{\lambda} - 8\gamma)(2\sqrt{\lambda}-\gamma) Y_2\left(2\sqrt{\gamma}\right) \\
  &-8\sqrt{\gamma}\sqrt{\lambda}(2\sqrt{\lambda}-\gamma) Y_3\left(2\sqrt{\gamma}\right)+O(\lambda^{3/2}).
\end{align*}

Applying once again Lemma \ref{lemma:exp} with
\[
\kappa=\frac{1}{2\sqrt\lambda},
\qquad G_0(\gamma)=-\frac{P(\gamma,0)}{Q(\gamma,0)}=- \frac{J_0\left(2\sqrt{\gamma}\right)}{Y_0\left(2\sqrt{\gamma}\right)} = -R_0(\gamma),
\]
we quickly arrive at
\begin{equation}\label{eq:ln-n-appr}
\lambda_n = 
\frac{\gamma^2}{4 n^2} \left( 1 + \frac2{\pi n} 
\Theta_{G_0}(\gamma)
+O(n^{-2})
\right) 
= \frac{\gamma^2}{4 n^2} \left( 1 - \frac2{\pi n} 
\Theta_{R_0}(\gamma)
+O(n^{-2})
\right) 
\end{equation}
as $n\to\infty$, thus proving the second part of Theorem \ref{thm:SAevs}.

\section{Proof of the asymptotic results of the non-self-adjoint operator}
\subsection{Eigenvalues and the Jost solutions}

\begin{lemma}
	The eigenvalues of \eqref{def:Agamma} are the zeroes of the determinant
\begin{equation}\label{eq:dt}
  M(\mu)=M_\gamma(\mu) = \phi'_{\mu}(\gamma,0)\, \phi_{-\mu}(\gamma,0) 
    + \phi'_{-\mu}(\gamma,0)\, \phi_{\mu}(\gamma,0).
\end{equation}
\end{lemma}

\begin{proof}
Suppose that $\mu\in\bC$ is an eigenvalue of $A_\gamma$, and that $g_\mu(x)\in L^2(\bR)$ is a corresponding eigenfunction. Then $g_\mu$ solves the differential equation
\[
-\frac{\dr^2}{\dr x^2}g_\mu(x)-\frac{\gamma}{1+|x|}g_\mu(x)=\sgn(x)\mu g_\mu(x).
\]
If $g_\pm$ denote the restrictions of $g_\mu$ on $\bR_+$ and $\bR_-$, then by integrability we must have
\[
g_+(x)=C_+ \phi_{\mu}(\gamma,x),\qquad g_-(-x)=C_-\phi_{-\mu}(\gamma,x),\qquad x\in\bR_+,
\]
where $\phi_{\mu}(\gamma,x)$ is the Jost solution \eqref{def:jost}. 

As an eigenfunction should be continuously differentiable at zero, we obtain
\[
\begin{cases}
C_+\phi_{\mu}(\gamma,0)-C_-\phi_{-\mu}(\gamma,0)&=0,\\
C_+\phi'_{\mu}(\gamma,0)+C_-\phi'_{-\mu}(\gamma,0)&=0,
\end{cases}
\]
which has a non-trivial solution if and only if $M_\gamma(\mu)=0$.
\end{proof}

\begin{remark}\begin{itemize}
	\item[{\rm (a)}] It follows from \eqref{def:jost} that if $\mu$ is real, either $\phi_\mu$ or $\phi_{-\mu}$ is not square integrable, and therefore $A_\gamma$ cannot have real eigenvalues.
	\item[{\rm (b)}] By \cite[Proposition 4.6]{behrndttrunk}  one can instead look for the eigenvalues of \eqref{def:op} as the zeroes of the m-function
\begin{equation}\label{eq:m}
  m_\gamma(\mu) = \frac{\phi'_{\mu}(\gamma,0)}{\phi_{\mu}(\gamma,0)} + \frac{\phi'_{-\mu}(\gamma,0)}{\phi_{-\mu}(\gamma,0)}.
\end{equation}
\tcR{The use of half-line m-functions is natural and has been already suggested elsewhere, and described in great generality for indefinite Sturm-Liovuille problems with turning point at $0$ in \cite{katr} (see also references therein).}
\item[{\rm (c)}] In what follows we assume that $\mu$ is in the upper half plane $\bC_+$ and look for the eigenvalues on the first quadrant. The final result will follow by symmetry (see Proposition \ref{prop:one}(a) and Proposition \ref{prop:one}(f)).
\end{itemize}
\end{remark}

\subsection{The determinant}\label{sec:det}

We can use \eqref{def:jost} and the known relations \cite[\S 13.3]{nist} between Kummer hypergeometric functions 
to rewrite \eqref{eq:dt} as
\begin{equation}\label{eq:dt1}\begin{split}
M(\mu) ={}& \frac{ 
  \gamma^2 \sqrt{-\mu}\; \er^{-\sqrt{-\mu}-\sqrt{\mu}} 
}{
  8\, \mu^{5/2}
}
\Bigg[ 
  \left(
    \gamma \sqrt{-\mu } + 2 \mu
  \right) 
  \wU\Big(-\frac{\gamma}{2 \sqrt{-\mu }},-1;\, 2 \sqrt{-\mu }\Big) 
  \wU\Big(-\frac{\gamma}{2 \sqrt{\mu }},0;\, 2 \sqrt{\mu }\Big) \\
  & + \left(
    2 \mu - \gamma \sqrt{\mu }
  \right)
  \wU\Big(-\frac{\gamma}{2 \sqrt{-\mu }},0;\, 2 \sqrt{-\mu }\Big)
  \wU\Big(-\frac{\gamma}{2 \sqrt{\mu }},-1;\, 2 \sqrt{\mu }\Big) \\
  & + 2 \mu \left(
    \sqrt{-\mu } + \sqrt{\mu } - 2
  \right)  
  \wU\Big(-\frac{\gamma}{2 \sqrt{-\mu }},0;\, 2 \sqrt{-\mu }\Big)
  \wU\Big(-\frac{\gamma}{2 \sqrt{\mu }},0;\, 2 \sqrt{\mu }\Big)
\Bigg].
\end{split}
\end{equation}

To find approximate solutions of $M(\mu) = 0$, we use the asymptotic formula \eqref{eq:approxUexp}.

Let us define for brevity
\begin{align*}
j_\nu &:= J_{\nu}(2\sqrt{\gamma}), \quad&
y_\nu &:= Y_{\nu}(2\sqrt{\gamma}), \\
K_+ &:= \cos \left(\frac{\gamma\pi}{2 \sqrt{\mu }}\right), \quad& K_- &:= \cosh \left(\frac{\gamma\pi}{2 \sqrt{\mu}}\right),\\
S_+ &:= \sin \left(\frac{\gamma\pi}{2 \sqrt{\mu }}\right),\quad& 
S_- &:= \sinh \left(\frac{\gamma\pi}{2 \sqrt{\mu }}\right), \\
T_+ &:= \tan \left(\frac{\gamma\pi}{2 \sqrt{\mu }}\right), \quad&
T_- &:= \tanh \left(\frac{\gamma\pi}{2 \sqrt{\mu }}\right).
\end{align*}

With these abbreviations and with the use of asymptotic formulae \eqref{eq:approxUexp}, equation \eqref{eq:dt} becomes
\begin{equation}\label{eq:dt2}
\begin{split}
\frac {\ir \sqrt{\gamma}}{\sqrt{\mu}}\, &
\Gamma\left(\frac{\ir \gamma}{2 \sqrt{\mu }}+1\right) 
\Gamma\left(\frac{\gamma}{2 \sqrt{\mu }}+1\right)
\Bigg\{
  \left( 1 - \frac{5 \sqrt{\mu }}{8 \gamma} \right) 
  \left( 2 \mu - \ir \gamma \sqrt{\mu } \right) 
  \left( j_1 K_+ + y_1 K_- \right) \\
&\quad\cdot 
  \left[
    \left( 1 + \frac{11 \ir \sqrt{\mu }}{8 \gamma}\right) 
    \left( j_2 S_+ + \ir y_2 S_- \right)
    + \frac{\sqrt{\mu } \left( y_3 S_- - \ir j_3 S_+ \right)}{\sqrt{\gamma}}
  \right] \\
& + \ir \left( 1 + \frac{5 \ir \sqrt{\mu }}{8 \gamma} \right)
  \left( j_1 S_+ + \ir y_1 S_- \right) 
  \Bigg[
    \left(2 \mu - \gamma \sqrt{\mu} \right) \\
&\quad\cdot 
    \left(
      \left( 1 - \frac{11 \sqrt{\mu }}{8 \gamma} \right) 
      \left( j_2 K_+ + y_2 K_- \right)
      + \frac{\sqrt{\mu } \left( j_3 K_+ + y_3 K_- \right)}{\sqrt{\gamma}}
    \right) \\
&\quad
    - \frac{\left(\frac{1}{8}-\frac{\ir}{8}\right) 
      \left(\sqrt{\mu }-(1+\ir)\right) \sqrt{\mu } 
      \left(8 \gamma-5 \sqrt{\mu }\right) 
      \left(j_1 K_+ + y_1 K_-\right)}{\sqrt{\gamma}}
  \Bigg] 
\Bigg\} = 0,
\end{split}
\end{equation}
where we have dropped the lower order terms.

Simplifying, writing $S_\pm = T_\pm K_\pm$, and collecting terms in $K_\pm$, we get

\begin{equation}\label{eq:dt3}
\begin{split}
&K_+ \Big\{
  (1+\ir) j_1^2 \sqrt{\gamma} \left(
    \sqrt{\mu } - (1+\ir)
  \right) \left(
    64 \ir \gamma^2 - 40 (1 + \ir) \gamma \sqrt{\mu } + 25 \mu 
  \right) \\
&\quad
  + T_- y_1 \left(
    -8 \ir \gamma^2 + (5+16 \ir) \gamma \sqrt{\mu } - 10 \mu
  \right) \left(
    j_2 \left(
      8 \gamma-11 \sqrt{\mu }
    \right) + 8 j_3 \sqrt{\gamma} \sqrt{\mu }
  \right) \\
&\quad 
  - j_1 \Big[
    -16 \ir j_3 \sqrt{\gamma} \sqrt{\mu } \left(
      4 (1 + \ir) \gamma^2-5 \gamma \sqrt{\mu } + 5 (1 - \ir) \mu 
    \right) \\
&\quad\quad
    + 2 j_2 \left(
      64 \gamma^3 - 128 (1 - \ir) \gamma^2 \sqrt{\mu } - 135 \ir \gamma \mu + 55 (1 + \ir) \mu^{3/2}
    \right) \\
&\quad\quad
    + T_- \left(
      8 \gamma-5 \sqrt{\mu }
    \right) \Big(
      (1+\ir) \sqrt{\gamma} \left(
        \sqrt{\mu } - (1+\ir)
      \right) \left(
        8 \gamma + 5 \ir \sqrt{\mu }
      \right) y_1 \\
&\quad\quad\quad      
      + \left(
        \gamma + 2 \ir \sqrt{\mu }
      \right) \left(
        8 \sqrt{\gamma} \sqrt{\mu } y_3 + y_2 \left(
          -11 \sqrt{\mu }+8 \ir \gamma
        \right)
      \right)
    \Big)
  \Big]
\Big\} \\
&- K_- \Big\{
  y_1 \Big[
    8 j_3 \sqrt{\gamma}\sqrt{\mu } \left(
      -8 \ir \gamma^2 + (16+5 \ir) \gamma \sqrt{\mu } - 10 \mu 
    \right) \\
&\quad
    + j_2 \left(
      64 \gamma^3-(40-216 \ir) \gamma^2 \sqrt{\mu } - (176+135 \ir) \gamma \mu + 110 \mu^{3/2}
    \right) \\
&\quad
    + (1+\ir) T_- \Big(
      \sqrt{\gamma} \left(
        \sqrt{\mu } - (1+\ir)
      \right) \left(
        64 \gamma^2 - 40 (1 - \ir) \gamma \sqrt{\mu } - 25 \ir \mu 
      \right) y_1 \\
&\quad\quad
      + 8 (1 - \ir) \sqrt{\gamma} \sqrt{\mu } y_3 \left(
        4 (1 + \ir) \gamma^2 - 5 \gamma \sqrt{\mu } + 5 (1 - \ir) \mu 
      \right) \\
&\quad\quad
      + y_2 \left(
        64 (1 + \ir) \gamma^3-256 \gamma^2 \sqrt{\mu }+135 (1 - \ir) \gamma \mu + 110 \ir \mu^{3/2}
      \right)
    \Big)
  \Big] \\
&\quad 
  + j_1 \left(
    8 \gamma+5 \ir \sqrt{\mu }
  \right) \Big(
    \left(
      k-2 \sqrt{\mu }
    \right) \left(
      y_2 \left(
        8 \gamma-11 \sqrt{\mu }
      \right) + 8 \sqrt{\gamma} \sqrt{\mu } y_3
    \right) \\
&\quad 
    + (1-\ir) \sqrt{\gamma} \left(
      \sqrt{\mu } - (1+\ir)
    \right) \left(
      8 \gamma-5 \sqrt{\mu }
    \right) y_1
  \Big)
\Big\} = 0.
\end{split}
\end{equation}

In what follows, we essentially  replicate the reasoning in Lemma \ref{lemma:exp}, but working to a higher accuracy. Introduce in the equation \eqref{eq:dt3} the ansatz
\begin{equation}\label{eq:ansatz}
\mu = \lambda + \nu \lambda^{3/2} + \eta \lambda^2,
\end{equation}
where $\lambda$ is, as before, the absolute value of an eigenvalue of the 
self-adjoint operator with either Dirichlet or Neumann boundary conditions. Now replace
back $K_-$, $T_-$ and $K_+$ with the corresponding expressions. The next step --- expanding the left-hand side
of the resulting equation in the Taylor series with respect to $\lambda$ 
around zero, --- is delicate.

First of all, observe that as $\lambda\to0$
\begin{align*}
\cos\left(\frac{\gamma\pi}{2\sqrt{\lambda + \nu \lambda^{3/2} + \eta \lambda^2}}\right)
={}& \cos\left(\frac{\gamma\pi}{2\sqrt{\lambda}} \left(  
  1 - \left(\frac{\nu \sqrt{\lambda}}2 + \left( \frac \eta2 - \frac{3\nu^2}8\right) \lambda + O(\lambda^{3/2}) \right)
\right)\right) \\
={}& \cos\left( \frac{\gamma\pi}{2\sqrt{\lambda}} \right)
     \cos\left(\frac{\gamma\pi}{4} \left( \nu + \left( \eta - \frac{3\nu^2}4 \right) \sqrt{\lambda} \right) \right) \\
     &+ \sin\left( \frac{\gamma\pi}{2\sqrt{\lambda}} \right)
     \sin\left(\frac{\gamma\pi}{4} \left( \nu + \left( \eta - \frac{3\nu^2}4 \right) \sqrt{\lambda} \right) \right) + O(\lambda),
\end{align*}
and similarly
\begin{align*}
\sin\left(\frac{\gamma\pi}{2\sqrt{\lambda + \nu \lambda^{3/2} + \eta \lambda^2}}\right)
={}& \sin\left( \frac{\gamma\pi}{2\sqrt{\lambda}} \right)
     \cos\left(\frac{\gamma\pi}{4} \left( \nu + \left( \eta - \frac{3\nu^2}4 \right) \sqrt{\lambda} \right) \right) \\
     &- \cos\left( \frac{\gamma\pi}{2\sqrt{\lambda}} \right)
     \sin\left(\frac{\gamma\pi}{4} \left( \nu + \left( \eta - \frac{3\nu^2}4 \right) \sqrt{\lambda} \right) \right) + O(\lambda).
\end{align*}

We want to derive  a similar expansion  for 
$\tanh\left(\frac{\gamma\pi}{2\sqrt{\lambda + \nu \lambda^{3/2} + \eta \lambda^2}}\right)$.
We use 
\[
\tanh(t_1 - t_2) = 
  \frac{\sinh(t_2) + \cosh(t_2)\tanh(t_1)}{\cosh(t_2) - \sinh(t_2)\tanh(t_1)}
\]
with $t_1 := \frac{\gamma\pi}{2\sqrt{\lambda}}$ and 
$t_2 := \frac{\gamma\pi}{4} \left( \nu + \left( \eta - \frac{3\nu^2}4 \right) \sqrt{\lambda} + O(\lambda) \right)$.

As
$\tanh(\pi \gamma/2\sqrt{\lambda}) = 1$ for $\lambda\to0$
modulo exponentially small terms,  we get 
(again up to exponentially small errors)
\[
\tanh\left(\frac{\gamma\pi}{2\sqrt{\lambda + \nu \lambda^{3/2} + \eta \lambda^2}}\right)
= \frac{\sinh(t_2) + \cosh(t_2)}{\cosh(t_2) - \sinh(t_2)} = -1,
\]
and \eqref{eq:dt3} reduces to an equation involving only 
$\tan\left( \frac{\gamma\pi}{2\sqrt{\lambda}} \right) $, 
$\tan\left(\frac{\gamma\pi}{4} \left( \nu + \left( \eta - \frac{3\nu^2}4 \right) \sqrt{\lambda} \right) \right)$ 
and powers of $\sqrt{\lambda}$. This is still, however, very hard to control.

\subsection{Complex eigenvalue curves}

We can now use our knowledge of the self-adjoint problem (see \eqref{eq:ev-d-sa-appr} and \eqref{eq:ev-n-as-approx}) to simplify \eqref{eq:dt3} further. 
For definiteness, suppose that $-\lambda$ is an eigenvalue of $T^D$.

Using the approximate identity \eqref{eq:ev-d-sa-appr}, obtained for the eigenvalues of the Dirichlet self-adjoint problem on the half line, we can reduce the already simplified \eqref{eq:dt3} to an equation involving only $\tan\left(\frac{\gamma\pi}{4} \left( \nu + \left( \eta - \frac{3\nu^2}4 \right) \sqrt{\lambda} \right) \right)$ and powers of $\sqrt{\lambda}$. Collecting the tangent terms, after some cumbersome but straightforward simplifications we arrive at
\begin{equation}\label{eq:withtan}
\tan\left(\frac{\gamma\pi}{4} \left( \nu + \left( \eta - \frac{3\nu^2}4 \right) \sqrt{\lambda} \right) \right)
= \frac{\tilde P(\lambda, \gamma)}{\tilde Q(\lambda,  \gamma)}
\end{equation}
where
\begin{align*}
\tilde P(\lambda, \gamma) &= -(4+4 \ir) \gamma + (12+7 \ir) \sqrt{\lambda }, \\
\begin{split}
\tilde Q(\lambda, \gamma) &= i(4+4 \ir) \gamma + (7+12 \ir) \sqrt{\lambda } \\ 
&\quad+2 \pi  \left( 4 \sqrt{\gamma } \left((1+\ir)\gamma -3 \sqrt{\lambda }\right) 
    \left( j_0 j_1 + y_0 y_1\right) 
    + 7 \sqrt{\lambda } \left( j_1^2 + y_1^2 \right) \right),
\end{split}
\end{align*}
and we have dropped terms of order $O(\lambda)$.  We have used the following standard relations \cite[Ch. 13]{nist} for the Bessel functions,
\begin{align*}
& J_{n+1}(2\sqrt{\gamma}) Y_n(2\sqrt{\gamma}) - J_n(2\sqrt{\gamma}) Y_{n+1}(2\sqrt{\gamma}) = \frac1{\pi \sqrt{\gamma}}, \\
& J_1(2\sqrt{\gamma}) - \sqrt{\gamma} J_2(2\sqrt{\gamma}) = \sqrt{\gamma} J_0(2\sqrt{\gamma}),\\
& Y_1(2\sqrt{\gamma}) - \sqrt{\gamma} Y_2(2\sqrt{\gamma}) = \sqrt{\gamma} Y_0(2\sqrt{\gamma}),
\end{align*}
in the simplifications.

Note that right-hand side of \eqref{eq:withtan} does not depend 
on $\nu$. We can now invert the tangent and solve with respect to  $\nu$, to find the coefficient for 
the $\lambda^{3/2}$ term in \eqref{eq:ansatz}. Expanding in $\lambda$ around $0$
and taking the leading term we get 
\begin{equation}\label{eq:CD}
\nu =: \Upsilon^-(\gamma) = \frac4{\pi \gamma}\arctan\left(
  \frac1{\ir-2\pi\sqrt{\gamma}\left(
    J_0(2\sqrt{\gamma})\, J_1(2\sqrt{\gamma}) + Y_0(2\sqrt{\gamma})\,Y_1(2\sqrt{\gamma})
  \right)}
\right).
\end{equation}
The contribution of the other terms in the expansion then forms a part of $\eta$, which we drop.

\begin{remark}
We can repeat the same procedure using the relation \eqref{eq:ev-n-as-approx} 
for the absolute value of a Neumann eigenvalue $\lambda$ 
as a starting point. In this case we obtain
\begin{equation}\label{eq:CN}
\nu =: \Upsilon^+(\gamma) =  \nu = \frac4{\pi \gamma}\arctan\left(
  \frac1{\ir+2\pi\sqrt{\gamma}\left( 
    J_0(2\sqrt{\gamma})\, J_1(2\sqrt{\gamma}) + Y_0(2\sqrt{\gamma})\,Y_1(2\sqrt{\gamma})
  \right)}
\right).
\end{equation}
\end{remark}

\begin{remark}
One can use standard asymptotic formulas for Bessel functions 
(see \cite[Chapter 10]{nist}) to observe 
\[
q(\gamma) \sim \begin{cases}
-\frac{\log(\gamma)}\pi & \mbox{if } \gamma\to 0 \\
\frac{1}{4 \sqrt{\gamma}}-\frac{3 \cos \left(4 \sqrt{\gamma}\right)}{512 \gamma} & \mbox{if } \gamma\to \infty
\end{cases}.
\]
Moreover, again using standard estimates and properties of Bessel functions and
their zeroes one can observe that
\[
|J_0(x)|,|J_1(x)|,|Y_0(x)|,|Y_1(x)| < \frac1{\sqrt{2}} \mbox{ for } x \ge 2,
\]
and the four functions are monotone for $x<2$ ($J_0$ increasing and bounded by $1$, 
the other ones decreasing, in particular the $Y$ are unbounded). In particular,
for any $\gamma\in(0,1)$ neither $\Upsilon^-(\gamma)$ and $\Upsilon^+(\gamma)$ nor 
their real and imaginary parts, vanish. Moreover $\Upsilon^\mp(\gamma)$ define curves in the complex
plane that diverge as $\gamma\to0$ and converge to $0$ as $\gamma\to+\infty$. 
\end{remark}

With $q(\gamma)$ as in \eqref{eq:q} we get  \eqref{eq:Ups^mp}.
Then the identity $\Im\Upsilon^-(\gamma) = \Im\Upsilon^+(\gamma)$ follows immediately using the standard relation between $\arctan$ and $\log$ (see Remark \ref{rem:comments}(a)), thus proving part (i) of Theorem \ref{thm:1}.

Proving part (ii) of Theorem \ref{thm:1} requires some extra work. 
First of all observe that, up to the errors of order $O(n^{-4})$, we have
\begin{align*}
\lambda^D_n - \lambda^N_n &= -\frac{\gamma^2}{2\pi n^3} 
  \left( 
    \Theta_{R_1}(\gamma) 
    - \Theta_{R_0}(\gamma) 
  \right) \\
  &=-\frac{\gamma^2}{4\pi n^3} \left( \ir\log\left(
      \frac{(J_1(2\sqrt{\gamma}) - \ir Y_1(2\sqrt{\gamma}))
        (J_0(2\sqrt{\gamma}) + \ir Y_0(2\sqrt{\gamma}))}
       {(J_1(2\sqrt{\gamma}) + \ir Y_1(2\sqrt{\gamma}))
        (J_0(2\sqrt{\gamma}) - \ir Y_0(2\sqrt{\gamma}))}\right)
    \right)\\
  &=-\frac{\gamma^2}{4\pi n^3} \left(
    \ir \log\left(\frac{q(\gamma) + \ir}{q(\gamma) - \ir}\right)
    \right),
\end{align*}
where $R_k(\gamma)$ are defined in Theorem \ref{thm:SAevs} and 
the first identity follows from 
the fact that zeroes of $Y_0$ and $Y_1$ are interlaced.
Similarly
\begin{align*}
\Upsilon^-(\gamma) - \Upsilon^+(\gamma) = 
  \frac{2\ir}{\pi\gamma}\log\left(\frac{2q(\gamma) - 2\ir}{2q(\gamma) + 2\ir}\right).
\end{align*}
Despite the appearance of the complex unity $\ir$ in the above formulae, all these expressions are in fact real!

To show that the two asymptotic formulae \eqref{eq:mu_first_quadrant} for $\mu_n$ coincide up to the lower order terms, we use Theorem \ref{thm:SAevs}, to obtain, as $n\to \infty$,
\[
\begin{split}
 &\left(\lambda^D_n + \Upsilon^-_n(\gamma)(\lambda^D_n)^{3/2}\right)  - \left(\lambda^N_n+\Upsilon^+_n(\gamma)(\lambda^N_n)^{3/2}\right) + O \left( \frac1{n^4} \right)\\
&\quad= 
  -\frac{\gamma^2 \ir}{4\pi n^3} \log\left(\frac{q(\gamma) + \ir}{q(\gamma) - \ir}\right) 
  + \frac{2\ir}{\pi\gamma}\log\left(\frac{2q(\gamma) + 2\ir}{2q(\gamma) - 2\ir}\right) 
    \left(\frac{\gamma^2}{4n^2}\right)^{3/2} + O\left( \frac1{n^4} \right) \\
&\quad= 
  \frac{\gamma^2 \ir}{4\pi n^3} \left( \log\left(\frac{q(\gamma) + \ir}{q(\gamma) - \ir}\right)
  - \log\left(\frac{q(\gamma) + \ir}{q(\gamma) - \ir}\right) \right) + O \left( \frac1{n^4} \right) \\
&\quad= 
  O \left( \frac1{n^4} \right),
\end{split}
\]
thus concluding the proof.

\section{Generalizations and other remarks}
\begin{figure}[h!]
\centerline{
\includegraphics[width=.66\linewidth]{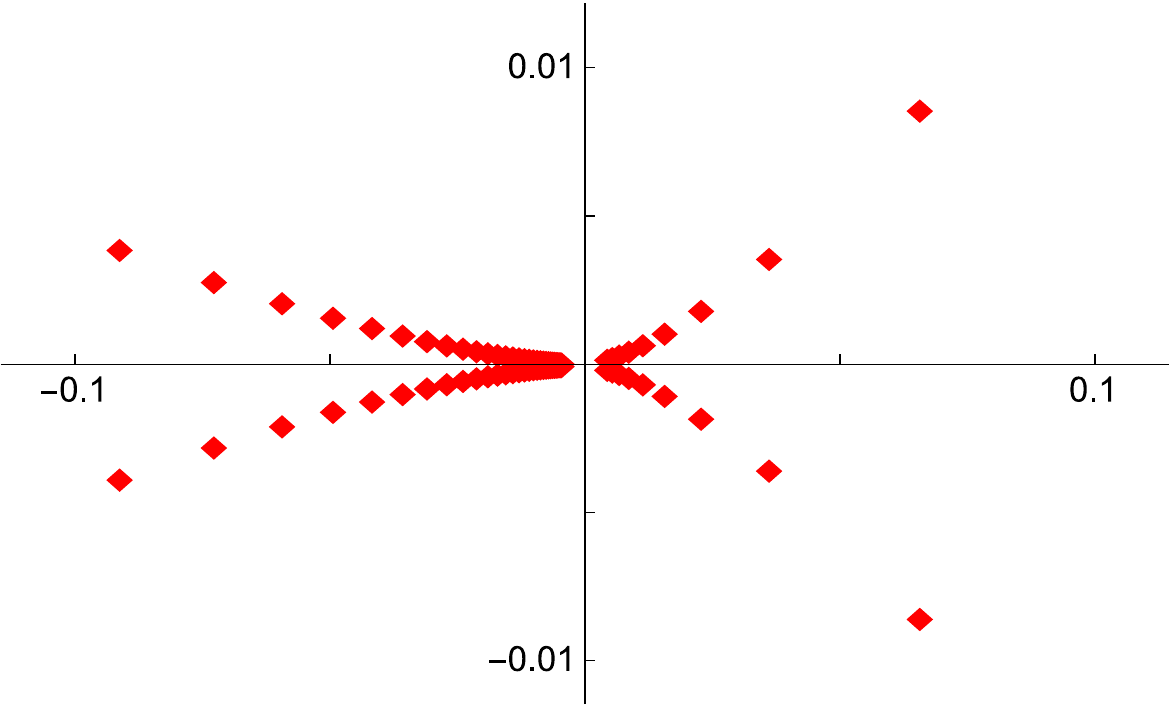}}
\caption{Approximated eigenvalues of $A(\gamma_+, \gamma_-)$ for $\gamma_-=1.5$, $\gamma_+=5$.}
\end{figure}
The procedure used to prove Theorem \ref{thm:1} can be repeated in a completely 
similar way to obtain a result for the operator
\[
A(\gamma_+, \gamma_-) = J T_{V}, 
\quad V(x) = \begin{cases}
\frac{\gamma_+}{1+|x|} & \mbox{if } x > 0\\
\frac{\gamma_-}{1+|x|} & \mbox{if } x < 0
\end{cases}, \quad  \gamma_+, \gamma_- \in \bR_+.
\]
In this case the m-function is of the form
\[
  M(\lambda) = \frac{\phi'_{\mu,\gamma_+}(0)}{\phi_{\mu,\gamma_+}(0)} 
    + \frac{\phi'_{-\mu,\gamma_-}(0)}{\phi_{-\mu,\gamma_-}(0)}.
\]
The curves in the upper (resp. lower) half plane are no more symmetric w.r.t. 
$\ir\bR$, however for the left quadrants and right quadrants we can extend Theorem \ref{thm:1}. 
The only difference is that now the $\Upsilon^-$ and $\Upsilon^+$ 
are now functions of both $\gamma_+$ and $\gamma_-$.

Let $\nu,\eta\in\bR_+$. Set
\begin{align*}
f_-(\nu,\eta) :={}&
  \frac{J_1^2(2\sqrt{\eta}) + J_0^2(2\sqrt{\eta})}
    {J_1^2(2\sqrt{\nu}) + J_0^2(2\sqrt{\nu})} 
  \Big(
  \ir 
  - \pi\sqrt{\nu} \left( J_0(2\sqrt{\nu})J_1(2\sqrt{\nu}) 
    + Y_0(2\sqrt{\nu})Y_1(2\sqrt{\nu}) \right) 
  \Big) \\
  & - \pi\sqrt{\eta} \left( J_0(2\sqrt{\eta})J_1(2\sqrt{\eta}) 
    + Y_0(2\sqrt{\eta})Y_1(2\sqrt{\eta}) \right) 
,\\
f_+(\nu,\eta) :={}&
  \frac{J_1^2(2\sqrt{\eta}) + J_0^2(2\sqrt{\eta})}
    {J_1^2(2\sqrt{\nu}) + J_0^2(2\sqrt{\nu})} 
  \Big(
  \ir 
  + \pi\sqrt{\nu} \left( J_0(2\sqrt{\nu})J_1(2\sqrt{\nu}) 
    + Y_0(2\sqrt{\nu})Y_1(2\sqrt{\nu}) \right) 
  \Big) \\
  & + \pi\sqrt{\eta} \left( J_0(2\sqrt{\eta})J_1(2\sqrt{\eta}) 
    + Y_0(2\sqrt{\eta})Y_1(2\sqrt{\eta}) \right) 
.
\end{align*}
Then the two factor multiplying the term $\Re\mu^{3/2}$ are given by
\begin{align*}
  \Upsilon^-(\gamma_+,\gamma_-) &:= \frac{4}{\pi}
    \begin{cases}
    \gamma_-^{-1}\;\arctan (1/f_-(\gamma_+, \gamma_-))  & \mbox{if } \Re\mu > 0\\
    \gamma_+^{-1}\;\arctan (1/f_-(\gamma_-, \gamma_+)) & \mbox{if } \Re\mu < 0
    \end{cases}, \\
  \Upsilon^+(\gamma_+,\gamma_-) &:= \frac{4}{\pi}
    \begin{cases}
    \gamma_-^{-1}\;\arctan (1/f_+(\gamma_+, \gamma_-))  & \mbox{if } \Re\mu > 0\\
    \gamma_+^{-1}\;\arctan (1/f_+(\gamma_-, \gamma_+))  & \mbox{if } \Re\mu < 0
    \end{cases}.
\end{align*}
One can immediately see that the asymmetry appearing w.r.t. $\ir\bR$ is 
reflected in the asymmetric dependence on $\gamma_+$ and $\gamma_-$. 

It is interesting to observe that for $\Re\mu > 0$ the effect of
$\gamma_-$ is much stronger than the one of $\gamma_+$ (the latter appears
only in the cotangent term, its contribution is bounded, while the former
additionally appears as an inverse prefactor). The situation 
is opposite when $\Re\mu < 0$. 

The expressions for $\Upsilon^{\mp}(\gamma_-,\gamma_+)$ are more involved than the ones for 
$\Upsilon^{\mp}(\gamma)$ but, as expected, they simplify to \eqref{eq:CD} and \eqref{eq:CN} 
for $\gamma_+ = \gamma_-$. As that case, it is possible to use the standard
results on Bessel functions to show that the two constants have non-zero real 
and imaginary part for any $\gamma_\pm > 0$.

To answer the general question posed in \cite{b13} for a wider class of potentials one would need good estimates of the Jost functions in a complex half ball containing the origin and the positive and negative real axis. 
To our knowledge, the best result of this kind is contained in a paper by Yafaev \cite{yafaev}. In that work, however, the author needed to exclude two cones containing the real axis for his estimates to hold. 
Additionally he could get only the first term in the asymptotic expansion, whereas for our result we would need at least the first two.

\numberwithin{equation}{section}
\numberwithin{theorem}{section}
\appendix
\section{Uniform asymptotic expansion of Kummer Hypergeometric functions}

We need to approximate
\begin{equation}\label{eq:UtoApprox}
  \wU\Big(-\frac{\gamma}{2 \sqrt{-\mu }},c;\, 2 \sqrt{-\mu }\Big),
  \qquad
  \wU\Big(-\frac{\gamma}{2 \sqrt{\mu }},c;\, 2 \sqrt{\mu }\Big),
  \quad c\in\{0,-1\},
\end{equation}
in the limit $\mu\to0$.
We use the theory developed in
\cite[Chapter 27]{temme}. By formula \cite[(27.4.85)]{temme}, as $a\to\infty$, and with $az$ bounded and $\Re(az)>0$, 
\begin{equation}\label{eq:UApprox}
\wU(-a, c;\, az) \sim \beta^{1-c}\, \Gamma(a+1) \er^{\frac12 a z}
  \left(
  C_{c-1}(\zeta) \sum_{n=0}^{\infty} \frac{A_n}{a^n} 
  + \beta C_{c-2}(\zeta) \sum_{n=0}^{\infty} \frac{B_n}{a^n} 
  \right),
\end{equation}
where $A_n$ and $B_n$ are defined by an iterative procedure, 
$\zeta = 2 \beta a$ and 
\begin{equation}\label{eq:cnuzeta}
C_\nu(\zeta) = \cos(\pi a) J_\nu(\zeta) + \sin(\pi a) Y_\nu(\zeta).
\end{equation}

On can immediately see that in our case
\begin{equation}\label{eq:az}
a = a_\pm = \frac{\gamma}{2\sqrt{\pm\mu}}, \qquad 
z = z_\pm = \frac{\gamma}{a_\pm^2}.
\end{equation} 
We will drop the subscript $\pm$ for the rest of the discussion.

Additionally we need an expression for $\beta$. This is defined 
in \cite[(27.4.36)]{temme} as 
\[
\beta = \frac12 (w_0 + \sinh(w_0))
\]
where $w_0 = 2 \arcsinh(\frac12 \sqrt{z})$, see \cite[(27.4.33)]{temme}. 
Equation \cite[(27.4.52)]{temme} gives an asymptotic expansion for $\beta$ as $z\to 0$:
\[
\beta^2 = z + \frac1{12} z^2 + O(z^3).
\]
By Taylor expansion we get
\begin{equation}\label{eq:beta}
\beta = \sqrt{z} \sqrt{1 + \frac1{12} z + O(z^2)} 
      = \sqrt{z} \left( 1 + \frac1{24} z + O(z^2) \right).
\end{equation}
Therefore
\[
\zeta = 2 \beta a = 2 \sqrt{\gamma} \left(1 + \frac \gamma{24} a^{-2} + O(a^{-4}) \right)
\]
and $\Re(az) = \Re(\gamma/a)$.

Observe that for $\gamma\in\bR_+$, $\Re(az) > 0$ iff $\Re(a)>0$.

The coefficients $A_0$ and $B_0$ also have explicit expressions that can be derived using some symmetry properties and L'H{\^ o}pital rule, see \cite[(27.4.74)]{temme}):
\begin{align*}
A_0 = \left(\frac{\beta}{2 \sin(\theta)}\right)^c \sqrt{\frac2\beta\tan\theta}\;\cos(c\theta) \\
B_0 = \left(\frac{\beta}{2 \sin(\theta)}\right)^c \sqrt{\frac2\beta\tan\theta}\;\frac{\sin(c\theta)}{\beta}
\end{align*}
where $\theta = - \frac12 \ir w_0$.

The computation of $A_n$ and $B_n$ for $n\geq0$ is quite involved, 
however we will need only $A_1$.
One can exploit the procedure to compute $A_0$ and $B_0$,
and the recursive definition of the coefficients to get a Taylor approximation 
in negative powers of $a$ for $c\in\{0,-1\}$.
We get
\begin{align}
\label{coeff:0}\mbox{if }c=0,{}& \quad A_0^0 = 1 + O(a^{-2}), \quad A_1^0 = -\frac{5}{16} + O(a^{-2}), \quad \beta B_0^0 = 0, \\
\label{coeff:-1}\mbox{if }c=-1,{}& \quad A_0^{-1} = 1 + O(a^{-4}), \quad A_1^{-1} = -\frac{11}{16} + O(a^{-2}), \quad \beta B_0^{-1} = -\frac{\sqrt{\gamma}}{2a} + O(a^{-3}).
\end{align}

With these, \eqref{eq:UApprox} can be re-written
\begin{equation}\label{eq:approxUexp}
\wU\left(-a, c;\, \frac{\gamma}{a}\right) \sim 
  \left(\frac{\sqrt{\gamma}}{a}\right)^{\frac{1-c}2}
  \Gamma(a+1)
  \, \er^{\frac{\gamma}{2a}}
  \left(
  \tilde C_{c-1}(a,\gamma)\, (\tilde A^c_0 + \tilde A^c_1) 
  + \tilde C_{c-2}(a,\gamma)\, \tilde B^c_0
  + O(a^{-2}) 
  \right)
\end{equation}
where
\begin{equation}\label{eq:Ctilde}
\tilde C_{\nu}(a,\gamma) :=
  \cos(\pi a) J_{\nu}(2\sqrt{\gamma}) + \sin(\pi a) Y_{\nu}(2\sqrt{\gamma})
\end{equation}
and $\tilde A^c_0$, $\tilde A^c_1$ and $\tilde B^c_0$ are obtained dropping the error term in the appropriate coefficient in \eqref{coeff:0} and \eqref{coeff:-1}.

\begin{figure}
\centerline{\includegraphics[width=0.7\linewidth]{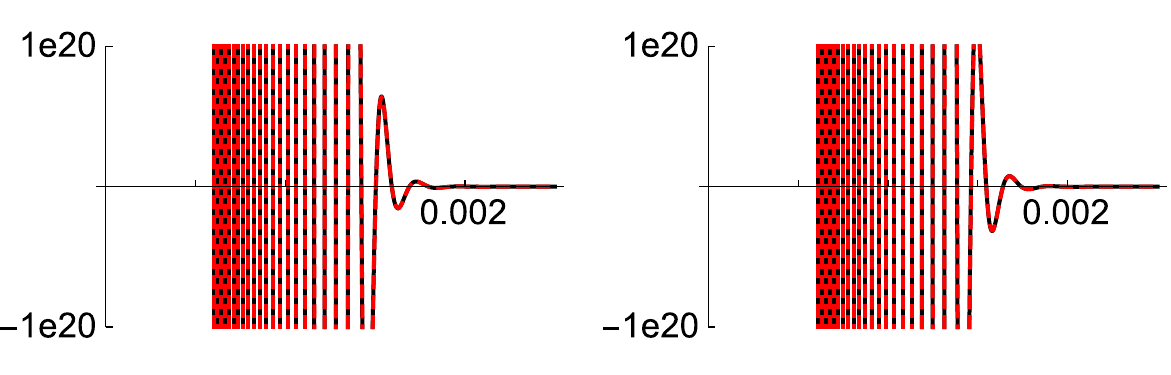}}
\caption{Plot of real part (left) and imaginary part (right) of $\wU\left(-\frac{\gamma}{2 \sqrt{-\mu }},c;\, 2 \sqrt{-\mu }\right)$ (black) and its approximation given by \eqref{eq:approxUexp} (dashed red) for small values of $\mu$ and $\gamma=2.5$.\label{fig:Uosc}}
\end{figure}

\begin{remark}
Here the error is in fact $O(\gamma/a^2)$, we may thus expect the improvement in the precision of the asymptotics when $\gamma \gg 1$.
\end{remark}

\begin{remark}[Validity of the expansion]
If we define
\[
t_1 = \beta + \pi \ir + \sqrt{(\beta + \pi \ir)^2 - \beta^2},
\]
then the asymptotic formula \eqref{eq:UApprox} is valid for
\begin{equation}\label{eq:conditions}
  -\arg t_1 - \frac\pi2 + \delta 
  \leq \arg a 
  \leq \arg t_1 + \frac\pi2 - \delta,
\end{equation}
and the same applies to $z$ (see \cite[Chapter 26.4.2]{temme}).

In our case
\[
t_1 \sim 
2 \pi \ir + 2 \beta + O(\beta^2).
\]
For $|a| \gg 1$, $\arg t_1$ is in the upper complex half plane. 
In particular this allows $a$ and $z$ to be in the closure of the 
first and fourth quadrant.
\end{remark}

\section{Aknowledgements}
The research of M. Seri has been supported by the EPSRC grant EP/J016829/1.
We would like to thank Jussi Behrndt for bringing the problem to our attention and stimulating discussions, and  
Adri Olde Daalhuis, Fritz Gesztesy, and Niko Temme for useful comments and references provided. We also acknowledge the hospitality of the Isaac Newton Institute for Mathematical Sciences in Cambridge, where this work was completed during the programme \emph{Periodic and Ergodic Spectral Problems}.


\end{document}